\documentclass[a4paper,11pt]{article}

\usepackage{amssymb,amsmath,amsfonts,latexsym,amsthm}
\usepackage[mathscr]{eucal}
\usepackage{srcltx}
\usepackage{verbatim}
\usepackage{graphicx}
\usepackage{amscd}

\small\normalsize

\newtheorem{theo}{Theorem}
\newtheorem{defi}[theo]{Definition}
\newtheorem{coro}[theo]{Corollary}
\newtheorem{prop}[theo]{Proposition}
\newtheorem{lemm}[theo]{Lemma}

\newcommand{\Z}{{\mathbb{Z}}}

\newcommand{\F}{{\mathbb{F}}}

\newcommand{\zero}{{\mathbf{0}}}

\newcommand{\iso}{\mbox{Iso}}

\title{On the number of nonequivalent propelinear extended perfect codes\thanks{This work was supported in part by the Spanish
MICINN under Grants  MTM2009-08435 and TIN2010-17358, and by the Catalan AGAUR under Grant 2009SGR1224.
The second author was supported by  the Grants RFBR 12-01-00448
and  12-01-31098. The work of the fourth author was partially
supported by the Grant RFBR 12-01-00631-a. \newline
$^1$J. Borges and J. Rif\`{a}  are with the
Department of Information and Communications Engineering,
                            Universitat Aut\`{o}noma de Barcelona,
                            08193-Bellaterra, Spain
                            (emails:~\{joaquim.borges,josep.rifa\}@autonoma.edu).\newline
$^2$I. Yu. Mogilnykh and F. I. Solov'eva   are with the Sobolev Institute
of Mathematics and Novosibirsk State University, Novosibirsk,
Russia (emails:~\{ivmog,sol\}@math.nsc.ru).}
}

\author{J. Borges$^1$, I. Yu. Mogilnykh$^2$, J. Rif\`{a}$^1$, F. I. Solov'eva$^2$}

\begin{document}

\maketitle

\begin{abstract} The paper proves that there exists an exponential  number of nonequivalent propelinear extended perfect binary codes
of length growing to infinity. Specifically, it is proved that all transitive extended perfect binary codes found by Potapov~\cite{Pot} are
propelinear. All such codes have small rank, which is one more than the rank of the extended Hamming code of the same length. We
investigate the properties of these codes and show that any of them
has a normalized propelinear representation.
\end{abstract}

\section{Preliminaries}

Let $E_{q}$ be a set of $q$ elements, where we distinguish one of
them and write it as $0$. We call \textit{words} the elements of the cartesian
product $E_q^n$. The word $(0,\ldots,0)$ is denoted by $\zero$. Given two words
$u=(u_1,u_2,\cdots, u_n), v=(v_1,v_2,\cdots,v_n) \in E_q^n$, the \textit{Hamming distance} $d(u,v)$ is the
number of positions where they differ. In some cases, when
we are interested in an algebraic structure inside $E_q$ we will
take the $q$-ary finite field $\F_q$
instead of $E_q$, with $q=p^m$ and $p$ prime. The action
of an isometry of $E_{q}^n$ can be presented as the action of a
permutation $\pi$ on the coordinate positions $\{1,\ldots,n\}$ followed by the
action of $n$ permutations $\sigma_{1},\ldots,\sigma_{n}$ of
$E_{q}$:
$$\pi(x_{1},\ldots,x_{n})=(x_{\pi^{-1}(1)},\ldots,x_{\pi^{-1}(n)}),$$
$$(\sigma_{1},\ldots,\sigma_{n})(x_{1},\ldots,x_{n})=(\sigma_{1}(x_{1}),\ldots,\sigma_{n}(x_{n})).$$

The permutation $\sigma = (\sigma_{1},\ldots,\sigma_{n})$ is called
a {\it multi-permutation}. The composition $\sigma\circ
\sigma'$ of two multi-permutations $\sigma$ and $\sigma'$ is
the following multi-permutation: $(\sigma_{1}\circ
\sigma'_{1},\ldots,\sigma_{n}\circ \sigma'_{n})$, where $\sigma_{i}\circ
\sigma'_{i}$ is the composition $\sigma_{i}\circ
\sigma'_{i}(x_i)=\sigma_{i}(
\sigma'_{i}(x_i))$, for any $i\in\{1,2,\ldots,n\}$.

 By $(\sigma;\pi)(x)$ we denote the image of $x$ under
an isometry $(\sigma;\pi)$ :

$$(\sigma;\pi)(x)=\sigma(\pi(x)).$$


A $q$-ary {\it code} $C$ of length $n$ is a subset of $E_{q}^n$. We
denote by $\iso(C)$ the largest subgroup of isometries
of $E_{q}^n$ that fix
the code $C$ and we call it the {\it isometry group} of the
code $C$.

\begin{defi}\label{q-ary_propelinear}
A $q$-ary code $C$ of length $n$ is called {\em propelinear} if for any codeword
$x$ there exists a coordinate permutation $\pi_x$ and
a multi-permutation $\sigma_x = (\sigma_{x,1},\ldots,\sigma_{x,n})$ satisfying:
\begin{itemize}
\item[(i)] for any $x\in C$ it holds  $(\sigma_{x};\pi_{x})(C)=C$ and  $(\sigma_{x};\pi_{x})(\zero)=x$,
\item[(ii)] if $y\in C$ and $z=(\sigma_{x};\pi_{x})(y)$, then:\\ $\pi_{z}=\pi_{x}\circ \pi_{y}$ and $\sigma_{z,i}=\sigma_{x,
i}\circ \sigma_{y,\pi^{-1}_{x}(i)},$ for any $i \in
\{1,\ldots,n\}$; or, equivalently,
$(\sigma_{z};\pi_{z})=(\sigma_{x};\pi_{x})(\sigma_{y};\pi_{y}).$
\end{itemize}
\end{defi}

A $q$-ary code is called {\em transitive} if the  isometry group of the code acts transitively on its codewords, i.\,e., the code satisfies the property $(i)$ in Definition \ref{q-ary_propelinear}. Transitive codes are studied in \cite{S2004,S2005}.

In the binary case, when $q=2$, taking the usual addition on
$E_2=\F_2$, the above definition is reduced to the following:

A binary code $C$ is {\em propelinear} if for each $x\in C$ there
exists a coordinate permutation $\pi_x$ such that:
\begin{itemize}
\item[(i)] $x+\pi_x(C)=C$;
\item[(ii)] if $x+\pi_x(y)=z$, then $\pi_z=\pi_x\circ \pi_y$, for any
$y\in C$.
\end{itemize}


As in the binary case, where we can define a group structure $\star$ on $C$
which is compatible with the Hamming distance, that is, such that
$d(x\star u,x\star v)=d(u,v)$, also in the $q$-ary case, given a $q$-ary propelinear code we define the operation $\star$ as
$$
x \star v = (\sigma_x;\pi_x)(v)\;\;\;\mbox{for any } x\in C,\;\;\mbox{for any }
v\in E_q^n.$$

In \cite{BR,Rif3,Rif1,Rif2}, properties of binary propelinear
codes are deeply studied. Linear codes and $\Z_2\Z_4$-linear codes are propelinear but, perhaps, one much more interesting example of propelinear code is the original Preparata code~\cite{Rif2} wich is not a $\Z_4$-linear code (although there is a $\Z_4$-linear code with the same parameters~\cite{kh}). In \cite{BRS2011,BMRS2011}, the
relations between classes of propelinear and transitive codes are
investigated. The problem of distinguishing these classes had been
open since 2006. In \cite{BMRS2011} it was established that these classes are different. In fact, it was found that the binary Best code of length 10 is transitive, but not
propelinear.

In this paper we establish a new lower bound $\frac{1}{8n^2\sqrt{3}}e^{\pi\sqrt{2n/3}}(1+o(1))$ on the number of nonequivalent propelinear extended perfect binary codes
of length $4n$ for $n$ going to infinity. This bound is obtained by showing propelinearity of transitive Potapov codes~\cite{Pot}, the rank of which is one more than the rank of the extended Hamming code of the same length. The previous lower bound on the number of nonequivalent propelinear extended perfect binary codes of length $n=2^m, m\geq 4$ was $\lfloor \log_2 (n/2)\rfloor ^2,$ see \cite{BRS2011,BMRS2011}. Therefore, despite the fact that the new class of propelinear codes is larger than the old class from \cite{BRS2011,BMRS2011}, it does not cover the old one, so the result \cite{BRS2011,BMRS2011} keeps current.    We
investigate in this paper the properties of new propelinear codes and show that any of them
has a normalized propelinear representation.

Now, we give a generalization of the most relevant
properties of propelinear codes to the $q$-ary case.

Let $C$ be a propelinear code; let $\Pi$ and $\Sigma$ be the sets of permutations assigned to the codewords of $C$ (Definition~\ref{q-ary_propelinear}) and let $\star$ be the afore defined operation in $C$. We will call $(C,\Pi,\Sigma,\star)$ the propelinear structure defined, which will be called $(C,\star)$ when we do not require any
information about the set of associated permutations.

The next lemmas are easy to prove from elementary group theory.

\begin{lemm}\label{previs}
Let $(C,\Pi,\Sigma,\star)$ be a $q$-ary propelinear code of length $n$.
\begin{itemize}
\item[(i)] Let $x\in C$ and $u,v\in E_q^n$. If $x\star u=x\star v$, then $u=v$.
\item[(ii)] The all-zeroes word $\zero$ is a codeword, $\zero\in C$.
\item[(iii)] For any codeword $x\in C$, there exists a unique codeword $x'\in C$ such that $x\star x'=\zero$.
\end{itemize}
\end{lemm}


\textbf{Note:} Not always the defined operation $\star$ can be  generalized in a proper way over all $E_q^n$. That is, from $x, y \in C$, $u\in E_q^n$ such that $x\star u = y \star u$, we can not necessarily have $x=y$.

\begin{lemm}
Let $(C,\Pi,\Sigma,\star)$ be a $q$-ary propelinear code. Then $C$ equipped with this operation $\star$ is a
group.
\end{lemm}

%
%
%
%

 Note that, apart from the group structure on $C$ given by the operation $\star$, there can exist
a lot of different group structures on a propelinear code, including
nonisomorphic ones (for binary case see \cite{BMRS2011}).

Clearly, $\zero$ is the identity element in $(C,\star )$ and we
denote by $x^{-1}$ the inverse element of the codeword $x$. Denote
by $Id_n$ the identity permutation over any set of cardinal $n$. Now, we can link
the coordinate permutations and the multi-permutations of inverse
elements.

\begin{lemm}
Let $(C,\star)$ be a $q$-ary propelinear code. Then,
\begin{itemize}
\item[(i)] The codeword $\zero$ has the identities as the associated coordinate permutation $\pi_\zero=Id_n$ and
multi-permutation $\sigma_\zero=(Id_q, Id_q,\ldots, Id_q)$, respectively.
\item[(ii)]
We have $\pi_{x^{-1}}=\pi_x^{-1}$ and
$\sigma_{x^{-1},i}=\sigma_{x,\pi_x(i)}^{-1}$, for any codeword $x\in C$ and any $i \in \{1,\ldots,n\}$.
\end{itemize}
\end{lemm}

%
%

Finally, as we said before, the action of $\star $ over $E_q^n$ is
Hamming distance compatible.

\begin{lemm}
Let $(C,\star )$ be a $q$-ary propelinear code. Then,
$$
d(x\star u,x\star v)=d(u,v)\;\;\;\mbox{for any } x\in C,\;\;\mbox{for any } u,v\in
E_q^n.
$$
\end{lemm}


\section{Isotopic propelinear MDS codes}
A $q$-ary
code of length $n$, satisfying the property $(i)$ in
Definition~\ref{q-ary_propelinear} with $\pi_x=Id_n$ for any $x$ in
the code is called an {\it isotopic transitive code}. A notion of isotopic transitivity was introduced by Potapov in \cite{Pot} and used for constructing an exponential number of nonequivalent transitive extended binary perfect codes of length $n$ as $n$ goes to infinity. We call a $q$-ary propelinear structure on a code $C$ of length $n$ {\em isotopic
propelinear}, if for any $x \in C$ it holds $\pi_x=Id_n$. If there is an isotopic propelinear structure on a code $C$, we call $C$ {\it isotopic propelinear}.

A $q$-ary code $C$ of length $m$ with minimum distance 2
of size $q^{m-1}$ is a kind of MDS code. All MDS-codes
we use in this paper are of this kind. A {\it quaternary
MDS code} is a code with $q=4$. A function
$f:E_{q}^{m-1}\rightarrow E_{q}$ is called a $(m-1)$-ary
{\it quasigroup} of order $q$ if $f(x_{1},\ldots,x_{m-1})\neq
f(y_{1},\ldots,y_{m-1})$ for any words $(x_{1},\ldots,x_{m-1})$
and $(y_{1},\ldots,y_{m-1})$ from $E_{q}^{m-1}$ that differs in
only one position.




It is known that there exists a one-to-one correspondence between
$(m-1)$-ary quasigroups of order $q$ and MDS $q$-ary codes of
length $m$. Given a $(m-1)$-ary quasigroup $f$ we can construct the
code $\{(x,f(x)):x \in E_{q}^{m-1}\}$.

In the rest of this section we  use $E_4=\{0,1,2,3\}$ as a set of four elements. Moreover, we use two different operations defined in $E_4$. First, we use $*$ to refer to the addition when we see the elements in $E_4$ as elements in $\Z_4$. Second, we use $\oplus$ to refer to the addition when we see the elements in $E_4$ as elements in $\Z_2\times \Z_2$ through the Gray map given by   $0\rightarrow (0,0), 1
\rightarrow (0,1), 2\rightarrow (1,1), 3\rightarrow (1,0)$.

Both next examples were used in~\cite{Pot} to construct extended perfect transitive codes.
\smallskip

{\bf Example 1.} Let us consider the function $x_1 \ast x_2$
from $E_{4}^{2}$ to $E_{4}$. From the correspondence between MDS codes and
quasigroups we have that $\{(x_1,x_2,x_1 \ast x_2): x_1,x_2\in E_{4}\}$ is a MDS code. It is straightforward to see that this code is an isotopic
propelinear code with the corresponding permutations $\sigma_{x,1}(y)=x_{1}*y$,
$\sigma_{x,2}(y)=x_{2}*y$, $\sigma_{x,3}(y)=x_{3}*y$ for any $y\in E_4$, where
$x_3=x_1 \ast x_2$.
\medskip

{\bf Example 2.} Let $x_{1}\oplus x_{2}$ be
the function from $E_{4}^{2}$ to $E_{4}$.
The corresponding MDS code is isotopic propelinear with  the following permutations $\sigma_{x,1}(y)=x_{1}\oplus y$,
$\sigma_{x,2}(y)=x_{2}\oplus y$, $\sigma_{x,3}(y)=x_{3}\oplus y$ for $y\in E_4$, where
$x_3=x_1 \oplus x_2$.

Potapov~\cite{Pot} proved isotopic transitivity of quaternary
MDS codes, obtained from an isotopic transitive MDS code $M$ and
the MDS code from Example 2, using the following concatenation
construction:
\begin{equation}\label{MDScomp}
\{(x_{1},\ldots,x_{i-1},y_{1},\ldots,y_{r},x_{i+1},\ldots,x_{m}):
y_{1}\oplus y_{2}\oplus \ldots\oplus y_{r}=x_{i}, \,\,
x=(x_{1},\ldots,x_m)\in M\},\end{equation} for some fixed $i$,
$1\leq i\leq m-1$, and for any $r=1,2,\ldots$

If the initial code corresponds to a
quasigroup $f$, that is,
$M=\{(x,f(x)):x \in E_{4}^{m-1}\}$ then the constructed code
corresponds to the following composition of the quasigroup $f$ and
the quasigroup from Example 2:
$$g(x_{1},\ldots,x_{i-1},y_{1},\ldots,y_{r},x_{i+1},\ldots,x_{m-1})=f(x_{1},\ldots,x_{i-1},y_{1}\oplus
y_{2}\oplus \ldots\oplus y_{r},x_{i+1},\ldots,x_{m-1}).$$

The main result of this section is Proposition~\ref{L1}, where we show that the constructed code is, in fact,
isotopic propelinear given that $M$ is isotopic propelinear. Below, without
restricting generality, $i$ is equal to 1.

First of all we recall two technical lemmas.

\begin{lemm}\cite{Pot}\label{lem_pot} Let $\varphi$ be a permutation on the
elements of $E_{4}$.
Then
$\varphi(a\oplus b)=\varphi(a)\oplus \varphi(b) \oplus \varphi(0).$
\end{lemm}

Given a permutation $\sigma$ on the elements of
$E_{4}$ and a word $y=(y_{1},\ldots,y_{r})$ in $E_{4}^{r}$ such
that $y_{1}\oplus \ldots \oplus y_{r}=\sigma(0)$ we define the
permutations $\tau_{y,1},\ldots,\tau_{y,r}$ in
$E_{4}$ in the following way:

\begin{equation}\label{tau}
\tau_{y,s}(\alpha)=\sigma(\alpha)\oplus y_{1} \oplus \ldots \oplus
y_{r} \oplus y_{s} = \sigma(\alpha)\oplus \sigma(0) \oplus y_{s}, \,\mbox{where $s\in \{1,2,\ldots,r\}.$}
\end{equation}

The above defined permutations satisfy the following statement:
\begin{lemm}\cite[Prop. 7]{Pot}\label{Proptau}
For any $x_{1},\ldots,x_{r}\in E_{4}$ we have $\tau_{y,1}(x_{1})\oplus\ldots\oplus \tau_{y,r}(x_{r})=\sigma(x_{1}\oplus\ldots\oplus x_{r}).$
\end{lemm}


\begin{prop}\label{L1} Let $(M,\Sigma,\star)$ be a quaternary isotopic propelinear MDS code of length
$m$ and
$$M'=\{(y_{1},\ldots,y_{r},x_{2},\ldots,x_{m}):(y_{1},\ldots,y_{r})\in
E_{4}^{r}, \,\, y_{1}\oplus y_{2}\oplus \ldots\oplus y_{r}=x_{1}, \,\, (x_{1},\ldots,x_{m})\in M\}.$$ Then $(M',\Delta,\star)$ is an isotopic
propelinear structure on the MDS code $M'$ with the multi-permutation $\delta_{z}=(\tau_{y,1},\ldots,\tau_{y,r},\sigma_{x,2},\ldots,\sigma_{x,m}),
$ assigned to the word $z=(y_{1},\ldots,y_{r},x_{2},\ldots,x_{m})$,
where $\tau_{y,s}$ is defined in (\ref{tau}) taken $\sigma_{x,1}$ as the permutation $\sigma$, for any $s\in \{1,2,\ldots,r\}$.
\end{prop}
\begin{proof}
It is easy to see that the code $M'$ has minimum distance 2,
length $m+r-1$ and size $4^{r+m-2}$, i.e., it is an MDS code over $E_{4}$.

By definition of a propelinear structure if a codeword $z$ of $M'$ is obtained from a
codeword $x$ of $M$ by replacing
the first coordinate $x_{1}$ with the sequence of elements $y_{1},\ldots,y_{r}$ from $E_4$,
such that $y_{1}\oplus\ldots\oplus y_{r}=x_{1}$, then
$\delta_{z}=(\tau_{y,1},\ldots,\tau_{y,r},\sigma_{x,2},\ldots,\sigma_{x,m})$.

The code $M'$ equipped with the permutations defined above was proved to be isotopic transitive, see \cite{Pot}. Now, we show that
 this structure is isotopic propelinear. In order to do so, following Definition~\ref{q-ary_propelinear}, we need to show that $\delta_{\delta_{z}(z')}=\delta_{z}\circ\delta_{z'}$, where $\delta_z, \delta_{z'},\delta_{\delta_z(z')}$ are the assigned permutations to the elements $z,z', \delta_z(z')$, respectively.

 Let $z=(y_{1},\ldots ,y_{r},x_{2},\ldots,x_{m})$ and $z'=(y'_{1},\ldots ,y'_{r},x'_{2},\ldots,x'_{m})$ be two arbitrary codewords of
 $M'$, i.e.
$$ y_{1}\oplus\ldots\oplus y_{r}=x_{1},$$
 \begin{equation}\label{y}
 y'_{1}\oplus\ldots\oplus y'_{r}=x'_{1},
 \end{equation}
where  $x=(x_{1},\ldots,x_{m}), x'=(x'_{1},\ldots,x'_{m})\in M$.
The permutations assigned to $z$ and $z'$ are:
$$\delta_{z}=(\tau_{y,1},\ldots,\tau_{y,r},\sigma_{x,2},\ldots,\sigma_{x,m})$$
$$\delta_{z'}=(\tau_{y',1},\ldots,\tau_{y',r},\sigma_{x',2},\ldots,\sigma_{x',m}).$$


 Consider
$\delta_{z}(z')=(\tau_{y,1}(y_{1}'),\ldots,\tau_{y,r}(y_{r}'),\sigma_{x,2}(x_{2}'),\ldots,\sigma_{x,m}(x_{m}')).$ By Lemma~\ref{Proptau} we have:
$\tau_{y,1}(y'_{1})\oplus\ldots\oplus\tau_{y,r}(y'_{r})=\sigma_{x,1}(y'_{1}\oplus
\ldots\oplus y'_{r})$. From this equality and (\ref{y}) we see
that $\delta_{z}(z')$ is obtained by substituting the first
coordinate of
$\sigma_{x}(x')=(\sigma_{x,1}(x'_{1}),\ldots,\sigma_{x,m}(x'_{m}))$
with the sequence of elements $\tau_{y,1}(y'_{1}),\ldots,\tau_{y,r}(y'_{r})$ and
$\tau_{y,1}(y'_{1})\oplus\ldots\oplus\tau_{y,r}(y'_{r})=\sigma_{x,1}(x'_{1})$. Therefore, $\delta_{z}(z')$ belongs to $M'$.
 From isotopic propelinearity of $(M,\Sigma,\star)$, we have that the permutation $\sigma_{x}\circ \sigma_{x'}$ is assigned to the codeword
 $\sigma_{x}(x')$ of $M$, so the multi-permutation
$\delta_{\delta_{z}(z')}$ coincides with $\delta_{z}\circ\delta_{z'}$ in each one of the $j$th positions, for $r+1\leq j \leq m+r-1$.

For the first $r$ positions, by the definition of $(M',\Delta,\star)$, we have
\begin{equation}\label{tau3}
\delta_{\delta_{z}(z'),s}(\alpha)=\tau_{\delta_{z}(z'),s}(\alpha)=\sigma_{x,1}\circ\sigma_{x',1}(\alpha)\oplus\tau_{y,1}(y'_{1})
\oplus\ldots\oplus\tau_{y,r}(y'_{r})\oplus\tau_{y,s}(y'_{s})
\end{equation} for $s=1,\ldots,r.$
It remains to prove that the permutation ($\ref{tau3}$) coincides with $\tau_{z,s}\circ\tau_{z',s}$ for
$s=1,\ldots,r$.
For any $s$ above,  using Lemma~\ref{lem_pot} and Lemma~\ref{Proptau}, the above equality (\ref{tau3}) comes to:
\begin{equation*}
\begin{split}
\tau_{\delta_{z}(z'),s}(\alpha) &=\sigma_{x,1}(\sigma_{x',1}(\alpha))\oplus \sigma_{x,1}(y_{1}'\oplus\ldots\oplus y'_r)\oplus \sigma_{x,1}(y'_s)\oplus\sigma_{x,1}(0)\oplus y_{s}\\
&= \sigma_{x,1}(\sigma_{x',1}(\alpha))\oplus \sigma_{x,1}(x_{1}'\oplus y'_s)\oplus y_{s}\\
&= \sigma_{x,1} \big(\sigma_{x',1}(\alpha)\oplus x'_1 \oplus y'_s \big) \oplus \sigma_{x,1}(0)\oplus y_{s}\\
&= \tau_{z,s}(\sigma_{x',1}(\alpha)\oplus \sigma_{x',1}(0) \oplus y'_s) \\
&= \tau_{z,s}(\tau_{z',s}(\alpha)).\\
\end{split}
\end{equation*}
\end{proof}

Potapov \cite{Pot} considered quasigroups of the following form:
$$f(x_{1},\ldots,x_{n-1})=(x_{1}\oplus\ldots\oplus
x_{i_{1}})\ast(x_{i_{1}+1}\oplus\ldots\oplus x_{i_{2}})\ast\ldots
\ast(x_{i_{m-2}+1}\oplus\ldots\oplus x_{n-1}),$$
where $1\leq i_{1}\leq \ldots \leq i_{m-1}\leq n-1$ (in throughout what follows, we denote this quasigroup with $f_{i_{1},\ldots,i_{m-2}}$),
and proved the transitivity property of any MDS code corresponding to a quasigroup of this type.
In this section we show the isotopic propelinearty of these MDS codes.

Indeed, let $M$ be the code $\{(y_{1},\ldots,y_{m-1},y_{1}\ast\ldots\ast
y_{m-1}):y_{j}\in E_{4}, j=1,2,\ldots, m-1\}$. This code is isotopic
propelinear with the permutation
$\sigma_{y}=(\sigma_{y,1},\ldots,\sigma_{y,m-1},\sigma_{y,m})$
assigned to the codeword
$y=(y_{1},\ldots,y_{m-1},y_{1}\ast\ldots\ast y_{m-1})$, where
$\sigma_{y,j}(\alpha)=\alpha\ast y_{j}$, for $1\leq j\leq m-1$ and
$\sigma_{y,m}(\alpha)=\alpha \ast y_{1}\ast \ldots\ast y_{m-1}$. In order to obtain the code
$M'=\{(x,f_{i_{1},\ldots,i_{m-2}}(x): x \in E_{4}^{n-1} \}$ we apply $m-1$
times the construction (\ref{MDScomp}) to every coordinate $j$,
$1\leq j\leq m-1$. By Proposition~\ref{L1}, the code $M'$
is isotopic propelinear. In other words, we obtain the following statement.


\begin{coro} \label{sigmaMDS}
Let $M'=\{(x,f_{i_{1},\ldots,i_{m-2}}(x)): x \in E_{4}^{n-1}\}$. Then there exists an isotopic propelinear structure $(M',\Sigma,\star)$, with
 the
multi-permutation $\sigma_{x}$ assigned to a codeword $x$ being such that
\begin{equation}
\sigma_{x,i_{j}+t}(\alpha)=(\alpha\ast
(x_{i_{j}+1}\oplus\ldots\oplus x_{i_{j+1}}))\oplus x_{i_{j}+t},
\end{equation}
for $1 \leq t \leq i_{j+1}-i_{j}$ and $0\leq j \leq m-2$, $i_{0}=0$.
\end{coro}

As a consequence of isotopic propelinearity of these code we have the same
lower bound for the number of nonequivalent isotopic propelinear
codes as the one in \cite{Pot} for the isotopic transitive codes.

\begin{coro}\label{number isotopic propelinear MDS}
There  exist at least
$\frac{1}{4(n-1)\sqrt{3}}e^{\pi\sqrt{2(n-1)/3}}(1+o(1))$
nonequivalent quaternary isotopic propelinear MDS codes of length
$n$, for $n$ going to infinity.
\end{coro}

\section{Propelinear extended perfect codes}

In this section we prove that binary extended perfect Phelps codes
\cite{P84} constructed from isotopic propelinear MDS codes are propelinear.

First of all, we give some additional notations from \cite{Pot} and prove some
necessary statements. Let $C_{0}$ be the binary extended  Hamming
code of length 4:

$$C_{0}=\{(0,0,0,0),(1,1,1,1)\}.$$

Take the elements in $E_4$ as the positions of the vectors $(v_1,v_2,v_3,v_4)\in \F_2^4$, where we understand that $0\in E_4$ points at the fourth coordinate. Hence $e_0$  means the all-zeroes vector in $\F_2^4$, except for the $4$th coordinate, which is $1$ and, in general, $e_a$  means the all-zeroes vector in $\F_2^4$, except for the $a$th coordinate which is $1$. Now, define the codes in $\F_2^4$:
\begin{equation}\label{car}
C_{a}^{r}=C_{0}+(1+r)e_{0}+e_{a}, \,\, \mbox{for } r\in\{0,1\}, \,\, a \in
E_{4}.
\end{equation}

Codes $\{C_{a}^{r}\}_{r=0,1;a \in E_{4}}$ give a partition of $\F_{2}^{4}$
into extended perfect codes and the codes $\{C_{a}^{r}\}_{r=0,a
\in E_{4}}$ gives a partition of the binary full even weight code into extended perfect codes.

All  extended perfect codes of length 4 can be represented as the cosets of $C_{0}$ and, more specifically, if $b$ is in $C_{a'}^{r'}$ then:
\begin{equation}\label{sum}
b+C_{a}^{r}=C_{a\oplus a'}^{r+r'}.
\end{equation}

Let $S_4$ be the symmetric group of permutations over $E_4$.
In \cite{Pot} it was shown that the action of a  permutation of coordinates on the partition $\{C_{a}^{r}\}_{r=0,a
\in E_{4}}$  can be represented as the action of a permutation of the cosets in this partition.

\begin{prop}\cite{Pot}\label{sigmapi}
For any $\sigma \in S_4$ there exists $\pi \in S_{4}$ such that
\begin{itemize}
\item[(i)] $C_{\sigma(a)}^{r}+e_{\sigma(0)}+e_{0}=\pi(C_{a}^{r})$,
\mbox{for all $a \in E_{4}$ and $r\in\{0,1\}$ },

\item[(ii)] The above defined permutation $\pi$ fixes the 4th coordinate, that is,  $\pi(4)=4$.
\end{itemize}

\end{prop}

 In \cite{Pot} it was proven that
$\sigma\rightarrow \pi$ is a mapping from $S_4$ onto the subgroup of $S_{4}$, fixing the element $4$. We now
show that this mapping is a homomorphism with the kernel consisting of
``linear" permutations, i.e., permutations $\sigma$ of type
$\sigma(\alpha)=\alpha \oplus b$, for some fixed $b \in E_{4}$ and for all $\alpha\in E_4$.

\begin{prop}\label{propsigmapihom}
\begin{itemize}

\item[(i)] Let $\sigma$, $\pi$, $\sigma'$, $\pi'$ be such that

\begin{equation}\label{sigma} C^{r}_{\sigma(\alpha)}+e_{\sigma(0)}+e_{0}=\pi(C_{\alpha}^{r}),
\mbox{ for all } \alpha \in E_{4} \mbox{ and } r\in\{0,1\};
\end{equation}

\begin{equation}\label{sigma1} C_{\sigma'(\alpha)}^{r}+e_{\sigma'(0)}+e_{0}=\pi'(C_{\alpha}^{r}),
\mbox{ for all } \alpha \in E_{4}, \mbox{ and } r\in\{0,1\}.
\end{equation}
 Then

\begin{equation}\label{sigmasigma1}
C^{r}_{\sigma(\sigma'(\alpha))}+e_{\sigma(\sigma'(0))}+e_{0}=\pi(\pi'(C_{\alpha}^{r})),
\mbox{ for all } \alpha \in E_{4}, r\in\{0,1\}.
\end{equation}
\item[(ii)] Let $\pi, \sigma$ satisfy (\ref{sigma}). Then $\pi=Id$ if and only if
$\sigma(\alpha)=\alpha \oplus b$, for some fixed $b \in E_{4}$.
\end{itemize}
\end{prop}

\begin{proof}~
\begin{itemize}
\item[(i)] Let $\pi''$ be such that

$$C_{\sigma(\sigma'(\alpha))}^{r}+e_{\sigma(\sigma'(0))}+e_{0}=\pi''(C_{\alpha}^{r})\mbox{ for any } \alpha \in \F_{4}, \mbox{ and any } r\in\{0,1\}.$$
We show that $\pi''=\pi\circ \pi'$. Consider the left side of the
equality (\ref{sigmasigma1}). Using (\ref{sigma}) and
(\ref{sigma1}) and that $\pi^{-1}$ fixes the 4th coordinate position we obtain
$$C^{r}_{\sigma(\sigma'(\alpha))}+e_{\sigma(\sigma'(0))}+e_{0}=C^{r}_{\sigma(\sigma'(\alpha))}+e_{\sigma(0)}+e_{0}+e_{\sigma(0)}+e_{\sigma(\sigma'(0))}=$$
$$\pi(C_{\sigma'(\alpha)}^{r})+e_{\sigma(\sigma'(0))}+e_{\sigma(0)}=\pi(\pi'(C_{\alpha}^{r})+e_{\sigma'(0)}+e_{0})+e_{\sigma(\sigma'(0))}+e_{\sigma(0)}=$$
$$\pi(\pi'(C_{\alpha}^{r}))+e_{\pi^{-1}(\sigma'(0))}+e_{0}+e_{\sigma(\sigma'(0))}+e_{\sigma(0)}.$$

Taking into account (\ref{car}), we see that showing
$e_{\pi^{-1}(\sigma'(0))}+e_{0}+e_{\sigma(\sigma'(0))}+e_{\sigma(0)}$
is in the code $C_{0}$ is enough to prove (i).

On one hand using (\ref{sigma}) we obtain
$$C_{\sigma(\sigma'(0))}^{r}+e_{\sigma(0)}+e_{0}=\pi(C_{\sigma'(0)}^{r})=\pi(C_{0}+e_{\sigma'(0)}+(1+r)e_{0})=$$
$$=C_{0}+e_{\pi^{-1}(\sigma'(0))}+(1+r)e_{0}.$$

On the other hand (\ref{car}) implies

$C_{\sigma(\sigma'(0))}^{r}+e_{\sigma(0)}+e_{0}=C_{0}+(1+r)e_{0}+e_{\sigma(\sigma'(0))}+e_{\sigma(0)}+e_{0}.$

Combining the last two equalities we obtain:
$$C_{0}+(1+r)e_{0}+e_{\pi(\sigma'(0))}=C_{0}+(1+r)e_{0}+e_{\sigma(0)}+e_{0}+e_{\sigma(\sigma'(0))},$$
which holds if and only if
$e_{\pi^{-1}(\sigma'(0))}+e_{0}+e_{\sigma(\sigma'(0))}+e_{\sigma(0)}$
is in $C_{0}$.

\item[(ii)] Let $\sigma(\alpha)=\alpha \oplus b$ for some $b \in E_{4}$. We
show that the permutation $\pi$ satisfying (\ref{sigma}) is the
identity permutation. Obviously it is true when $b=0$, so in
through out what follows we consider $b$ to be nonzero. From
(\ref{sigma}) using  (\ref{car}) we can see that it  is
enough to show that $e_{0}+e_{b}+e_{\alpha}+e_{\alpha \oplus b}$
is in $C_{0}$ for any $\alpha$ in $E_{4}$. Obviously, for $\alpha$
being equal to 0 or $b$ the vector
$e_{0}+e_{b}+e_{\alpha}+e_{\alpha \oplus b}$ is the all-zeroes vector.
For $\alpha$ being different from $0$ and $b$, the element $\alpha\oplus b$ is
different from $b$, $0$ and from $\alpha$ because $b$ is nonzero.
Therefore all elements $0,b,\alpha, \alpha\oplus b$ are different, in
other words $e_{0}+e_{b}+e_{\alpha}+e_{\alpha \oplus b}$ is the
all-ones vector. So, the permutation $\pi$ is the identity.

Vice versa, if $\pi$ is the identity, by the
first isomorphism theorem in group theory there are no other permutations on
$E_{4}$ with $\pi=Id$ than the ones described before. Indeed, the order of the kernel of the
homomorphism $\sigma\rightarrow \pi$ from $S_4$ onto the group of permutations in $S_{4}$ fixing the
fourth element is equal to 4.
\end{itemize}
\end{proof}

Now, consider  the Phelps concatenation construction \cite{P84}, see also \cite{Z76}:

\begin{equation}\label{Phelps}
C=\bigcup_{(h_{1},\ldots,h_{n})\in
H}\bigcup_{(a_{1},\ldots,a_{n})\in M}
C_{a_{1}}^{h_{1}}\times\ldots\times C_{a_{n}}^{h_{n}},
\end{equation}
where $H$ is an extended Hamming code of length $n$, $M$ is a
quaternary MDS code of length $n$ and codes
$C_{a_{i}}^{h_{i}}, i=1,\ldots,n,$ are defined in (\ref{car}).
Using the construction (\ref{Phelps}), Potapov in \cite{Pot} found
a large class of transitive codes taking $M$ being isotopic
transitive. 
These  MDS codes  correspond to quasigroups
\begin{equation}  f(x_{1},\ldots,x_{n-1})=(x_{1}\oplus\ldots\oplus
x_{i_{1}})\ast(x_{i_{1}+1}\oplus\ldots\oplus x_{i_{2}})\ast\ldots
\ast(x_{i_{m-2}+1}\oplus\ldots\oplus x_{n-1}),\end{equation} for
any $i_{1},\ldots,i_{m-2}$, such that $1\leq i_{1}<\ldots <
i_{m-2} < n-1$. In the previous section we proved that all these
isotopic transitive 
MDS codes are isotopic propelinear. Now, we show that
all Potapov's transitive extended perfect binary codes are propelinear too.

\begin{theo}\label{Phelps.theo} Let $M$ be a quaternary isotopic propelinear MDS code of length
$n$, $H$ be a binary extended Hamming code of length $n$. Then,
the code
$$ C=\bigcup_{(h_{1},\ldots,h_{n})\in
H}\bigcup_{(a_{1},\ldots,a_{n})\in M}
C_{a_{1}}^{h_{1}}\times\ldots\times C_{a_{n}}^{h_{n}}
$$
is a binary propelinear extended perfect code of length $4n$.
\end{theo}
\begin{proof} Let $(M,\Sigma,\star)$ be an isotopic propelinear
structure on the code $M$. Let
$\sigma_{a}=(\sigma_{a,1},\ldots,\sigma_{a,n})$ be the
multi-permutation assigned to a codeword $a=(a_{1},\ldots,a_{n})$ of $M$. For any $i\in
\{1,\ldots,n\}$, let $\pi_{a_{i}}$ be the permutation defined by
Proposition \ref{sigmapi} when $\sigma$ is equal to
$\sigma_{a,i}$:

$$C_{\sigma_{a,i}(b)}^{r}+e_{a_{i}}+e_{0}=\pi_{a_{i}}(C_{b}^{r}), \,\, \mbox{for any b} \in E_{4}, \,\, \mbox{and any} \,\, r\in\{0,1\}.$$

To every codeword $c$ in the class
$C_{a_{1}}^{h_{1}}\times\ldots\times C_{a_{n}}^{h_{n}},$ where
$(h_{1},\ldots,h_{n})\in H$ we assign the permutation
$\pi_{a}=(\pi_{a_{1}},\ldots,\pi_{a_{n}})$ acting on $4n$
coordinates in the following way: if $x=(x_{1},\ldots,x_{n})$ is a
word of length $4n$ such that $x_{i}$ is a word of length 4
for any $i$, then
$\pi_{a}(x_{1},\ldots,x_{n})=(\pi_{a_{1}}(x_{1}),\ldots,\pi_{a_{n}}(x_{n}))$.
In \cite{Pot} it is proved that the code $C$ with these
permutations is transitive. We now show that it is propelinear too.

Let $c \in C_{a_{1}}^{h_{1}}\times\ldots\times
C_{a_{n}}^{h_{n}}$;  $c'\in C_{a'_{1}}^{h'_{1}}\times\ldots\times C_{a'_{n}}^{h'_{n}}$ and let  $\pi_{a}$ and $\pi_{a'}$ be the permutations
assigned to the codewords $c$ and $c'$, respectively. To show that $C$ is propelinear it is enough to
show that  the permutation assigned to the class
$c+\pi_{a}(C_{a'_{1}}^{h'_{1}}\times\ldots\times
C_{a'_{n}}^{h'_{n}})$ is $\pi_{a}\circ \pi_{a'}$.

Let us find  more convenient representation for the class
$c+\pi_{a}(C_{a'_{1}}^{h'_{1}}\times\ldots\times
C_{a'_{n}}^{h'_{n}})$. By the definition of $\pi_{a}$ we have the
following equalities:

$$\pi_{a}(C_{a'_{1}}^{h'_{1}}\times\ldots\times
C_{a'_{n}}^{h'_{n}})=\pi_{a_{1}}(C_{a'_{1}}^{h'_{1}})\times\ldots\times
\pi_{a_{n}}(C_{a'_{n}}^{h'_{n}})=$$
$$(C_{\sigma_{a,1}(a'_{1})}^{h'_{1}}+e_{a_{1}}+e_{0})\times\ldots\times
(C_{\sigma_{a,n}(a'_{n})}^{h'_{n}}+e_{a_{n}}+e_{0}).$$
Since $c \in C_{a_{1}}^{h_{1}}\times\ldots\times C_{a_{n}}^{h_{n}}$,
 from the last equality and using (\ref{sum}) we obtain

$$c+\pi_{a}(C_{a'_{1}}^{h'_{1}}\times\ldots\times
C_{a'_{n}}^{h'_{n}})=C_{\sigma_{a,1}(a'_{1})}^{h'_{1}+h_{1}}\times\ldots\times
C_{\sigma_{a,n}(a'_{n})}^{h'_{n}+h_{n}}.$$

Thus, we have  to find the permutation corresponding to the
codewords of the class
$C_{\sigma_{a,1}(a'_{1})}^{h'_{1}+h_{1}}\times\ldots\times
C_{\sigma_{a,n}(a'_{n})}^{h'_{n}+h_{n}}$. By propelinearity of
$M$, the multi-permutation assigned to
$(\sigma_{a,1}(a'_{1}),\ldots,\sigma_{a,n}(a'_{n}))\in M$
is
$\sigma_{a}\circ\sigma_{a'}=(\sigma_{a,1}\circ\sigma_{a',1},\ldots,
\sigma_{a,n}\circ\sigma_{a',n})$. Finally, from Proposition
\ref{propsigmapihom} and the definitions of permutations on $C$, we
obtain that $\pi_{a}\circ\pi_{a'}$ is the permutation assigned to the
codewords of the class
$C_{\sigma_{a,1}(a'_{1})}^{h'_{1}+h_{1}}\times\ldots\times
C_{\sigma_{a,n}(a'_{n})}^{h'_{n}+h_{n}}=c+\pi_{a}(C_{a'_{1}}^{h'_{1}}\times\ldots\times
C_{a'_{n}}^{h'_{n}})$.
 \end{proof}

And finally considering MDS codes corresponding to quasigroups of type
$$f_{i_{1},\ldots,i_{m-2}}(a_{1},\ldots,a_{n-1})=(a_{1}\oplus\ldots\oplus
a_{i_{1}})\ast(a_{i_{1}+1}\oplus\ldots\oplus a_{i_{2}})\ast\ldots
\ast(a_{i_{m-2}+1}\oplus\ldots\oplus a_{n-1})$$ and applying the results of the previous section we obtain:

\begin{coro} There exist  at least
$\frac{1}{8n^2\sqrt{3}}e^{\pi\sqrt{2n/3}}(1+o(1))$  nonequivalent propelinear extended perfect binary codes of length $4n$, for $n$ going to infinity.
\end{coro}

\section{Normality}
The concept of binary normalized propelinear codes was introduced
in \cite{BMRS2011}.
 A propelinear structure on a binary code $C$
is called {\it normalized} if the codewords of the same coset of
the code $C$ by the kernel have the same assigned permutation.

In this section we analyze the propelinear structure defined in
Theorem \ref{Phelps.theo}, when MDS code corresponds to the
quasigroup $$f_{i_{1},\ldots,i_{m-2}}(a_{1},\ldots,a_{n-1})=(a_{1}\oplus\ldots\oplus
a_{i_{1}})\ast(a_{i_{1}+1}\oplus\ldots\oplus a_{i_{2}})\ast\ldots
\ast(a_{i_{m-2}+1}\oplus\ldots\oplus a_{n-1}).$$ We show that the structure is normalized if and only if $m$ is
odd. For even $m$ it is not normalized, however we can find an exponential number of propelinear representations of
the Phelps codes, which are normalized propelinear.



By $Ker(M)$, the {\it kernel} of an arbitrary MDS code $M$ over
$E_{4}$ we mean the collection of all its codewords $a$ such that
$$a\oplus M=M.$$
We begin  with describing the kernel of the MDS and Phelps
codes.

\begin{prop}\label{propMDSPh}
Let $M$ be a MDS code of length $n$, $H$ be an extended Hamming code of length $n$,
$$C=\bigcup_{h\in H}\bigcup_{a\in M}
C_{a_{1}}^{h_{1}}\times\ldots\times C_{a_{n}}^{h_{n}}.$$ Then
a codeword $c$  from the code $C_{a'_{1}}^{h'_{1}}\times\ldots\times
C_{a'_{n}}^{h'_{n}}$ belongs to $Ker(C)$ if and only if the word $a'=(a'_{1},\ldots,a'_{n})$ belongs to
$Ker(M)$.
\end{prop}
\begin{proof}
Let $c$ be a codeword in $C_{a'_{1}}^{h'_{1}}\times\ldots\times
C_{a'_{n}}^{h'_{n}}$. We know that $(h'_1,h'_2,\ldots, h'_n) +H =H$, so
$$
C=\bigcup_{h\in H}\bigcup_{a\in M}
C_{a_{1}}^{h_{1}+h'_1}\times\ldots\times C_{a_{n}}^{h_{n}+h'_n}.
$$

From (\ref{sum}) we have:

$$c+C=c+\bigcup_{h\in
H}\bigcup_{a\in M} C_{a_{1}}^{h_{1}+h'_1}\times\ldots\times
C_{a_{n}}^{h_{n}+h'_n}=\bigcup_{h\in H}\bigcup_{a\in M}C_{a_{1}\oplus a'_{1}}^{h_{1}}\times\ldots\times
C_{a_{n}\oplus a'_{n}}^{h_{n}}.$$ Therefore, $c$ is in $Ker(C)$ if and only if
$a'$ is in $Ker(M)$.
\end{proof}


From this point and further on we represent a codeword of a MDS
code corresponding to a quasigroup $f$ as $(a,f(a))$.

\begin{prop}\label{kerMDS}
Let $f$ be a (n-1)-quasigroup of order 4 and $M$ the MDS code $M=\{(a,f(a)): a \in
E_{4}^{n-1}\}$. The codeword $(a',f(a'))$ belongs to $Ker(M)$ if and only if
$f(a\oplus a')=f(a)\oplus f(a')$, for all $a \in E_{4}^{n-1}$.

\end{prop}
\begin{proof}
If $(a',f(a'))\in Ker(M)$ then $(a',f(a')) \oplus M=M$, so, for all $a\in E_4^{n-1}$ we have $(a'\oplus a, f(a')\oplus f(a)) \in M$. So $(a'\oplus a, f(a')\oplus f(a))$ and $(a\oplus a', f(a\oplus a'))$ both belong to $M$, therefore, $f(a)\oplus f(a')=f(a\oplus a')$ and vice versa.
\end{proof}

Now, we focus on the case when MDS code corresponds to the
quasigroup
$$f_{i_{1},\ldots,i_{m-2}}(a_{1},\ldots,a_{n-1})=(a_{1}\oplus\ldots\oplus
a_{i_{1}})\ast(a_{i_{1}+1}\oplus\ldots\oplus a_{i_{2}})\ast\ldots
\ast(a_{i_{m-2}+1}\oplus\ldots\oplus a_{n-1}).$$

We need a technical lemma. In through out what follows,
$u^{-1}$ denotes the inverse element of $u$ in the group
$(E_{4},\ast)$.

\begin{lemm}~\label{func}

\begin{itemize}
\item[(i)] For all $u\in E_4$ it is true $u'\oplus u = \left\{ \begin{array}{ll}
u' \ast u & \mbox{for } u'\in \{0,2\}; \\
u'\ast u^{-1} & \mbox{for } u'\in \{1,3\}.
\end{array}\right.$
\item[(ii)] There is no $u'\in E_{4}$ such that
the equality
$$u'\oplus (u\ast v)=u'\ast u^{-1}\ast v$$ holds for any $u$ and $v$ from $E_{4}$.
\end{itemize}
\end{lemm}

\begin{proof}
The first statement follows directly from the definitions of operations $\oplus$ and
$\ast$.

Let us prove the second statement. If $u'$ belongs to $\{0,2\}$, then by the first statement we have  $u' \oplus (u\ast
v)=u'\ast u\ast v$ for any $u$ and $v$ in $E_{4}$, but for $u \in
\{1,3\}$ we have: $u'\ast u\ast v\neq u'\ast u^{-1}\ast v$.

If $u'$ belongs to $\{1,3\}$, then by the first statement we obtain  $u' \oplus
(u\ast v)=u'\ast u^{-1}\ast v^{-1}$ for any $u$ and $v$ in
$E_{4}$, but for $v \in \{1,3\}$ we have $u'\ast u^{-1}\ast
v^{-1}\neq u'\ast u^{-1}\ast v$.
\end{proof}

\medskip

We now describe the kernel of a particular MDS code.
\begin{prop}\label{propp}
Let $M=\{(a,a_{1}\ast\ldots\ast a_{m-1}):a=(a_1,\ldots,a_{m-1})
\in E_{4}^{m-1}\}$.
 Then

 $$Ker(M)=\left\{%
\begin{array}{ll}
    \{(a',a'_{1}\ast\ldots\ast a'_{m-1})\mid & \hbox{$a'\in \{0,2\}^{m-1}$}\}, \mbox{ if m is odd,} \\
\{(a',a'_{1}\ast\ldots\ast a'_{m-1})\mid & \hbox{$a'\in \{0,2\}^{m-1}\cup \{1,3\}^{m-1}$}\}, \mbox{ if m is even.} \\
    \end{array}%
\right.     $$
\end{prop}
\begin{proof}
By Proposition \ref{kerMDS} it is true $(a',a'_{1}\ast\ldots\ast
a'_{m-1})\in Ker(M)$ if and only if
\begin{equation} \label{kerp1}
(a'_{1}\oplus a_{1})\ast\ldots \ast(a'_{m-1}\oplus
a_{m-1})=(a'_{1}\ast\ldots\ast a'_{m-1})\oplus
(a_{1}\ast\ldots\ast a_{m-1})\mbox{ for any } a \in E_{4}^{m-1}.
\end{equation}
Let the first $k$ coordinates of $a'$ be from $\{0,2\}$ and the
last $m-1-k$ coordinates be from $\{1,3\}$. Now, from Lemma \ref{func}
we can express the operation $\oplus$ in (\ref{kerp1}) by the
operation $\ast$: $(a'_{1}\oplus a_{1})\ast\ldots
\ast(a'_{m-1}\oplus a_{m-1})=a'_{1}\ast a_{1}\ast\ldots \ast
a'_{k}\ast a_{k}\ast a'_{k+1}\ast a^{-1}_{k+1}\ast \ldots\ast
a'_{m-1}\ast a^{-1}_{m-1},$ so the condition  (\ref{kerp1}) of
belonging to the kernel of $M$ can be rewritten as:

$$(a'_{1}\ast\ldots\ast a'_{m-1})\oplus (a_{1}\ast\ldots\ast a_{m-1})=$$
$$a'_{1}\ast\ldots\ast a'_{m-1}\ast a_{1}\ast \ldots \ast a_{k}\ast a^{-1}_{k+1}\ast \ldots \ast a^{-1}_{m-1},
\mbox{for any } a \in E_{4}^{m-1}.$$
    Let $u'$ be equal to $a'_{1}\ast\ldots\ast a'_{m-1}$; $u$ be equal to $a_{1}\ast\ldots\ast
    a_{k}$ and  $v$ be equal to $a_{k+1}\ast\ldots\ast a_{m-1}$. Using
    these notations, the property of belonging to the kernel is
    equivalent to:
\begin{itemize}
\item[(i)] If $k=0$, then $a'\in \{1,3\}^{m-1}$ and $u'\oplus v= u' * v^{-1}$ implying, by Lemma \ref{func}, that $u'\in \{1,3\}$. Therefore $m$ is even.
\item[(ii)] If $k=m-1$, then $a'\in \{0,2\}^{m-1}$ and $u'\oplus u = u'*u$ implying, by Lemma \ref{func}, that $u'\in \{0,2\}$, which is true for any $m$.
\item[(iii)] If $0<k<m-1$, then $u'\oplus (u*v)=u'*u*v^{-1}$, which is impossible, again according to Lemma \ref{func}.
\end{itemize}

%
\end{proof}

For the general case we have the following description for the kernels:
\begin{theo}\label{T2}
Let $M=\{(a,f_{i_{1},\ldots,i_{m-2}}(a)): a \in E_{4}^{n-1}\}$ be a MDS code. Then $(a,f_{i_{1},\ldots,i_{m-2}}(a))$ belongs to $Ker (M)$ if and only if  the word of partial sums
$$(\oplus_{j=1}^{i_{1}}a_{j},\oplus_{j=i_{1}+1}^{i_{2}}a_{j},\ldots,\oplus_{j=i_{m-2}+1}^{n-1}a_{j})$$
belongs to $\{0,2\}^{m-1}$ for odd $m$ and to $\{0,2\}^{m-1}\cup
\{1,3\}^{m-1}$ for even $m$.
\end{theo}
\begin{proof}
Directly from Proposition \ref{propp}.
\end{proof}
From Theorem \ref{T2} and Proposition \ref{propMDSPh} we
obtain:

\begin{coro}\label{kersize}
Let $C$ be the code obtained by  Phelps construction
$$C=\bigcup_{h\in H}\bigcup_{(a,f_{i_{1},\ldots,i_{m-2}}(a))\in M}
C_{a_{1}}^{h_{1}}\times\ldots\times C_{a_{n-1}}^{h_{n-1}}\times
C_{f_{i_{1},\ldots,i_{m-2}}(a_{1},\ldots,a_{n-1})}^{h_{n}}.$$

 If $m$ is odd
then $|Ker(C)|=2^{3n-2-log_{2}(n)}$ and if $m$ is even $|Ker(C)|=2^{3n-1-log_{2}(n)}$.
\end{coro}

\begin{theo}\label{PhelpsNormal}
Let $$C=\bigcup_{h\in H}\bigcup_{(a,f_{i_{1},\ldots,i_{m-2}}(a))\in M}
C_{a_{1}}^{h_{1}}\times\ldots\times C_{a_{n-1}}^{h_{n-1}}\times
C_{f_{i_{1},\ldots,i_{m-2}}(a_{1},\ldots,a_{n-1})}^{h_{n}},$$ and let $(C,\Pi,\star)$ be the propelinear structure on $C$, defined
in Theorem \ref{Phelps.theo}. Then

(i) $(C,\Pi,\star)$ is normalized, if $m$ is odd;

(ii) $(C,\Pi,\star)$ is not normalized and there exist at least $2^{n-2}$ different normalized propelinear structures on $C$, if $m$ is even.


\end{theo}

\begin{proof}
Consider a codeword $(a',f_{i_{1},\ldots,i_{m-2}}(a'))\in Ker(M)$ of the MDS code
$M'=\{(a,f_{i_{1},\ldots,i_{m-2}}(a)): a\in E_{4}^{n-1}\}$. Let the
multi-permutation
$$(\sigma_{a',1},\ldots,\sigma_{a',n-1},\sigma_{a',n})$$
be assigned to the word $(a',f_{i_{1},\ldots,i_{m-2}}(a'))$ of $M'$. Let $t$ be such that
$i_{s}+1\leq t \leq i_{s+1}$, for $s=0,\ldots, m$, $i_{0}=0,
i_{m-1}=n-1$. Then, by the definition of the propelinear structure
on $M'$, see Corollary \ref{sigmaMDS}, we have
$$\sigma_{a',t}(\alpha)=((a_{i_{s}+1}\oplus\ldots\oplus a_{i_{s+1}})\ast \alpha)\oplus
a_{t}\oplus a_{i_{s}+1}\oplus\ldots\oplus a_{i_{s+1}},$$
$$\sigma_{a',n}(\alpha)=(a_{1}\oplus\ldots\oplus
a_{i_{1}})\ast(a_{i_{1}+1}\oplus\ldots\oplus a_{i_{2}})\ast\ldots
\ast(a_{i_{m-2}+1}\oplus\ldots\oplus a_{n-1})\ast \alpha.$$

By Theorem \ref{T2}, $a_{i_{s}+1}\oplus\ldots\oplus a_{i_{s+1}}$
belongs to $\{0,2\}$ or $\{1,3\}$ simultaneously for any
$s=0,\ldots,m-2$.

Let $a_{i_{s}+1}\oplus\ldots\oplus a_{i_{s+1}}$ be from
$\{0,2\}$. Then, by Lemma \ref{func},
$\sigma_{a',t}(\alpha)=(a_{i_{s}+1}\oplus\ldots\oplus
a_{i_{s+1}})\oplus \alpha \oplus a_{t}\oplus
a_{i_{s}+1}\oplus\ldots\oplus a_{i_{s+1}}=\alpha\oplus a_{t}$.
 Therefore,
according to the assignment of permutations to the codewords of
$C$ (see Theorem \ref{Phelps.theo}) and by Proposition \ref{propsigmapihom} (ii),  the
permutation
$\pi_{(a',f_{i_{1},\ldots,i_{m-2}}(a'))}=(\pi_{(a',f_{i_{1},\ldots,i_{m-2}}(a')),1},\ldots,\pi_{(a',f_{i_{1},\ldots,i_{m-2}}(a')),n})$
assigned to a codeword $c$ in the class
$C_{a'_{1}}^{h'_{1}}\times\ldots\times
C_{a'_{n-1}}^{h'_{n-1}}\times C_{f_{i_{1},\ldots,i_{m-2}}(a')}^{h'_{n}}$ is the identity. Taking into account the description of kernel given in Theorem \ref{T2},
the structure is normalized for odd $m$.

Let $a_{i_{s}+1}\oplus\ldots\oplus a_{i_{s+1}}$ be from the set
$\{1,3\}$. Then, by Lemma \ref{func},
$\sigma_{a',t}(\alpha)=((a_{i_{s}+1}\oplus\ldots\oplus
a_{i_{s+1}})\oplus \alpha^{-1} \oplus a_{t}\oplus
a_{i_{s}+1}\oplus\ldots\oplus a_{i_{s+1}}=\alpha^{-1}\oplus
a_{t}$. It is easy to see that the permutation $\sigma_{a',t}(\alpha)=\alpha^{-1}\oplus a_{t}$
 is not a permutation of ``linear type", i.e. cannot be expressed as
$\alpha\oplus u$ for some fixed $u$ from $E_{4}$. So, according to
the assignment of permutations to the codewords of $C$ (see
Theorem \ref{Phelps.theo}), and again by Proposition \ref{propsigmapihom}(ii), the
permutation
$\pi_{(a',f_{i_{1},\ldots,i_{m-2}}(a'))}=(\pi_{(a',f_{i_{1},\ldots,i_{m-2}}(a')),1},\ldots,\pi_{a',f_{i_{1},\ldots,i_{m-2}}(a'),n})$
assigned to the codeword $c$ in
$C_{a'_{1}}^{h'_{1}}\times\ldots\times
C_{a'_{n-1}}^{h'_{n-1}}\times C_{f_{i_{1},\ldots,i_{m-2}}(a')}^{h'_{n}}$ is not the
identity.

Hence, for even $m$, one half of the codewords of the kernel have
assigned the identity permutation, and the second half have assigned the same non-identity permutation.

Note that in case of even $m$, there exist several normalized propelinear structures.
According to the proof of the Theorem~\ref{Phelps.theo},
every codeword is assigned a permutation of the form
 $\pi_{a}=(\pi_{a_{1}},\ldots,\pi_{a_{n}})$, where $\pi_{a_i}\in S_4$
fixes one coordinate (see Proposition \ref{sigmapi}), for any
$i=1,\ldots,n$. Hence, $\pi_{a_i}$ could only have order 2 or 3, for all
$i=1,\ldots,n$, so $\pi_a$ has order 2, 3 or 6.
However, orders 3 and 6 can not occur. Given a propelinear code
$(D,\Pi,\star)$ with $|D|$ being a power of 2, we have that
$|\Pi=\{\pi_{x}:x \in C\}|$ is also a power of 2. Indeed, the map $x \mapsto \pi_x$ is
group homomorphism from the group $D$ onto the permutation group $\Pi$.
Therefore, for any codeword $x$, the order of $\pi_x$ is a power
of 2 as well. We conclude that all permutations assigned to the
codewords of $C$ are involutions and, in this case, $C$ has $|\{\pi_{x}:x \in Ker(C)\}|^{log(|C|)-dim(Ker(C))}$ different
normalized propelinear structures (a proof of this fact is given in
\cite{BMRS2011}). Substituting the value of $dim(Ker(C))$ obtained from Corollary \ref{kersize}, we obtain $2^{n-2}$ different structures.

\end{proof}

\end{document}